\documentclass[12pt]{amsart}

\usepackage{amssymb,amsfonts,amsmath,amsopn,amstext,amscd,color,latexsym,mathrsfs,skak,stackrel,url,verbatim,amsthm}
\usepackage{mathabx}
\usepackage[utf8]{inputenc}
\usepackage[all, cmtip]{xy}
\usepackage{mathtools}
\usepackage{enumitem}
\usepackage[active]{srcltx}
\usepackage{cite}
\usepackage[pdftex, colorlinks=true,linkcolor = blue, citecolor = blue, linktoc=all]{hyperref}
\usepackage{soul}
\usepackage{cleveref}
\usepackage{quiver}
\makeatletter
\DeclareRobustCommand
  \myvdots{\vbox{\baselineskip4\p@ \lineskiplimit\z@
    \hbox{.}\hbox{.}\hbox{.}}}
\makeatother

\theoremstyle{plain}
\newtheorem{thm}{\bf Theorem}[section]
\newtheorem{lem}[thm]{\bf Lemma}
\newtheorem{cor}[thm]{\bf Corollary}
\newtheorem{prop}[thm]{\bf Proposition}

\theoremstyle{remark}

\newtheorem{rem}[thm]{\bf Remark}

\theoremstyle{definition}
\newtheorem{defn}[thm]{\bf Definition}

\numberwithin{equation}{thm}

\newcommand{\bbC}{{\mathbb C}}

\newcommand{\bbH}{{\mathbb H}}

\newcommand{\bbP}{{\mathbb P}}
\newcommand{\bbQ}{{\mathbb Q}}

\newcommand{\bbZ}{{\mathbb Z}}

\newcommand{\cA}{{\mathcal A}}

\newcommand{\cC}{{\mathcal C}}

\newcommand{\cE}{{\mathcal E}}

\newcommand{\cL}{{\mathcal L}}

\newcommand{\cM}{{\mathcal M}}

\newcommand{\cO}{{\mathcal O}}

\newcommand{\cT}{{\mathcal T}}

\newcommand{\cZ}{{\mathcal Z}}

\newcommand{\rD}{{\rm D}}

\newcommand{\rH}{{\rm H}}

\newcommand{\sA}{{\mathscr A}}

\newcommand{\sC}{{\mathscr C}}

\newcommand{\sM}{{\mathscr M}}

\newcommand{\oZ}{\overline{Z}}

\newcommand*{\sheafhom}{\mathrm{H}\kern -.5pt om}

\DeclareMathOperator{\Perv}{Perv}

\DeclareMathOperator{\Ram}{Ram}
\DeclareMathOperator{\CC}{CC}
\DeclareMathOperator{\id}{id}
\DeclareMathOperator{\tr}{tr}
\DeclareMathOperator{\Supp}{Supp}

\newcommand{\csM}{\overline{\sM}}
\newcommand{\csA}{\overline{\sA}}

\title[Log Riemann-Hurwitz inequality]{A Higher dimensional log Riemann--Hurwitz inequality and rigidity of covers}
\author{Donu Arapura}
\address{Department of Mathematics, Purdue University,
150 N. University Street, West Lafayette, IN 47907, U.S.A.}
\email{arapura@purdue.edu}
\author{Chikako Mese}
\address{Department of Mathematics, Johns Hopkins University, 3400 N. Charles Street, Baltimore, MD 21218, U.S.A.}
\email{cmese1@jhu.edu}
\author{Deepam Patel}
\address{Department of Mathematics, Purdue University,
150 N. University Street, West Lafayette, IN 47907, U.S.A.}
\email{patel471@purdue.edu}
\date{July 2026}

\date{\today}

\begin{document}

\begin{abstract}
Let $\oZ$ be a smooth projective variety with an simple normal crossing divisor $D$ such that $\Omega_{\oZ}^1(\log D)$ is nef. We prove that if $Y\subset Z:=\overline Z\setminus D$ is a smooth
closed subvariety of dimension $n$ and $P$ is a perverse sheaf on
$Y$ of generic rank $r$, then
$
  \chi(Y,P)\geq r(-1)^n\chi(Y).
$
In particular, $\chi(Y,P)>0$ if $P$ has full support and
$\chi(Y)\neq0$.
Applying this to the trace-zero part of a finite direct image
yields a logarithmic Riemann--Hurwitz inequality: any finite surjective morphism
$f\colon X\to Y$ of degree $d$ with $X$ smooth satisfies
$(-1)^n\chi(X)\ge d\,(-1)^n\chi(Y)$, the difference being an explicit sum of
nonnegative intersection numbers. When $(-1)^n\chi(Y)>0$ this forces any such
$f$ with $\chi(X)=\chi(Y)$ to be an isomorphism. We verify the nef hypothesis for subvarieties of semiabelian
varieties, and for
$\overline{\mathscr M}_{g,n}$---the moduli of curves, obtaining in particular
that every finite surjective self-morphism of a moduli space of curves with
level structure is an isomorphism.

\end{abstract}

\maketitle
\tableofcontents
\section{Introduction}

The purpose of this note is twofold. Our  primary goal  is to  give a higher dimensional generalization of an inequality for Euler characteristics for curves. We will say more about this  below. The secondary goal is to give a relatively simple proof of a refinement of the inequality
\[
  \chi^{\mathrm{orb}}(\sM_{g,n},P)\geq0,
\]
for perverse sheaves $P$, obtained by two of the authors in \cite{ArapuraPatel}. The refinement asserts that the above inequality is strict when $P$ has full support on the moduli stack of complex genus-$g$ curves with $n$ marked points, provided $2g-2+n>0$ (see Theorem~\ref{thm:strictpositivityMgn}).
The proof hinges on a fact stated without proof in  \cite{Litt} that the logarithmic cotangent bundle of
$\csM_{g,n}$ is nef. We include a proof of this statement as well. These results are applied to rigidity of covers of various moduli spaces. Below, we work over the field of complex numbers.

For a finite morphism $f:X\to Y$ of degree $d$ between smooth
projective curves, the Riemann--Hurwitz formula is
\[
  2g(X)-2=d\bigl(2g(Y)-2\bigr)+\deg R_f,
\]
where $R_f$ is the ramification divisor. We may re-write this as an equality of Euler characteristics:
$$-\chi(X) = d(-\chi(Y))  + \deg R_f.$$ This formula can be generalized to the setting of smooth
quasiprojective curves by incorporating the points at infinity
into the canonical divisors. In particular, for $X$,$Y$ smooth quasiprojective one obtains the logarithmic form of the previous formula:
\[
  -\chi(X)=d\bigl(-\chi(Y)\bigr)+\deg R_f^{\mathrm{int}},
\]
where $R_f^{\mathrm{int}}$ is the ramification divisor on $X$ lying over
the open curve $Y$.  In particular,
\[
  -\chi(X)\geq d\bigl(-\chi(Y)\bigr).
\]

In  this note, we prove a higher-dimensional analogue of this inequality of Euler characteristics in the setting where the
target has nef logarithmic cotangent bundle, and more generally with coefficients in perverse sheaves.  Throughout, all
varieties are complex algebraic varieties, and $\chi$ denotes the
topological Euler characteristic of the associated analytic space. For a
bounded constructible complex $K$ (of $\bbQ$-vector spaces) on $Y$, we write
\[
  \chi(Y,K):=\sum_i(-1)^i\dim_{\bbQ} \rH^i(Y,K)
\]
where $\rH^i(Y,K)$ denotes hypercohomology
A simple
normal-crossing divisor will be abbreviated as an SNC divisor.

\begin{thm}[Logarithmic Euler--Hurwitz inequality]
\label{thm:main}
Suppose that $\oZ$ is a smooth projective variety with an SNC divisor $D$ such that $\Omega_{\oZ}^1(\log D)$ is nef.
Let $Z= \oZ\setminus D$.
Given a smooth  connected closed subvariety $Y\subset Z$ of complex
dimension $n$ 
and a finite surjective morphism of degree $d$
\[
  f:X\longrightarrow Y,
\]
 where $X$ is smooth and connected,   we have
\begin{equation}
\label{eq:main-inequality}
  (-1)^n\chi(X)\geq d(-1)^n\chi(Y).
\end{equation}

\end{thm}

With the above notation, we define the
\emph{Euler ramification defect} of $f$ by
\[
  \Ram(f)
  :=(-1)^n\chi(X)-d(-1)^n\chi(Y).
\]
Thus Theorem~\ref{thm:main} says that $\Ram(f)\geq0$.  When $n=1$,
$\Ram(f)$ is exactly the degree of the ramification divisor lying over
the open target.


\begin{rem}
    The proof gives slightly more, namely that there are 
integer multiplicities $m_\alpha\ge 0 $ and  conical subvarieties
$\overline\Lambda_\alpha\subset T^*_{\overline Z}(\log D)$  such that $\Ram(f)$ is given by the intersection number
\begin{equation}
\label{eq:main-formula}
  \Ram(f)
  =
  \sum_\alpha m_\alpha
  \bigl\langle \overline\Lambda_\alpha,\overline Z\bigr\rangle,
\end{equation}
and every term on the right-hand side is nonnegative.
\end{rem}

Two immediate consequences are:

\begin{cor}[Degree bound]
\label{cor:degree-bound}
Under the hypotheses of Theorem~\ref{thm:main}, assume that
$(-1)^n\chi(Y)>0$.  Then
\[
  d\leq
  \frac{(-1)^n\chi(X)}{(-1)^n\chi(Y)}.
\]
\end{cor}

\begin{cor}[Euler-characteristic rigidity]
\label{cor:rigidity}
Under the hypotheses of Theorem~\ref{thm:main}, assume that
\[
  (-1)^n\chi(Y)>0
  \qquad\text{and}\qquad
  \chi(X)=\chi(Y).
\]
Then $d=1$, and $f$ is an isomorphism.  In particular, the conclusion
holds if $f$ is a homotopy equivalence.
\end{cor}

There are several interesting examples to which these results can be applied. We discuss a few  of them here.
A semiabelian variety is an extension of an abelian variety by an algebraic torus.
Given a smooth variety $Y$,
there exists a semiabelian variety $Alb(Y)$ and morphism $\alpha:Y\to Alb(Y)$,
and moreover $\alpha$ is universal among all such morphisms \cite{Fujino}.

\begin{cor}\label{cor:Alb}
    If $\alpha:Y\to Alb(Y)$ is a closed embedding, then \eqref{eq:main-inequality} holds for any degree $d$ finite morphism $f:X\to Y$ with $X$ smooth and connected.

\end{cor}

Along the way to the proof of Corollary~\ref{cor:Alb} we recover, with a short proof, the
theorem of Franecki--Kapranov \cite{FraneckiKapranov} that perverse sheaves on a
semiabelian variety have nonnegative Euler characteristic. The key observation (Proposition \ref{prop:semiab}) is the existence of a canonical log compactification with nef log cotangent bundle for any semiabelian variety.

Let $\sM_{g,n}$ denote the Deligne--Mumford stack of smooth
$n$-pointed curves of genus $g$.  
Given a finite-index torsion-free subgroup $\Gamma$ of the mapping
class group $\Gamma_{g,n}$, we obtain
a smooth quasiprojective variety $\mathcal{M} ={}_\Gamma \sM_{g,n}$, which can be viewed as the fine moduli space of pointed curves with suitable level structures. It is given analytically as the quotient of Teichm\"uller space $\Gamma\backslash \cT_{g,n}$.
This admits a  finite \'{e}tale morphism
\[
  q:\cM\longrightarrow\mathscr M_{g,n}
\]

\begin{cor}
\label{cor:moduliofcurves}
Assume $2g-2+n>0$. Given a cover
\[
  q:\mathcal M\longrightarrow\mathscr M_{g,n},
\]
as above,  the following statement holds.
If
\[
  f:X\longrightarrow\mathcal M
\]
is a finite surjective morphism of degree $d$, where $X$ is smooth
and connected, then
\[
  (-1)^{3g-3+n}\chi(X)
  \geq
  d\,(-1)^{3g-3+n}\chi(\mathcal M).
\]
Consequently, if
$
  \chi(X)=\chi(\mathcal M),
$ then $f$ is an isomorphism. 
\end{cor}

Let $\Gamma\subset Sp_{2g}(\bbZ)$ be a torsion free subgroup of finite-index, we have a quasi-projective variety  
$\cA = {}_\Gamma \sA_g$ given analytically as a quotient of the Siegel upper half-space $\Gamma\backslash \bbH_g$.
This admits an \'etale morphism
$$q: \cA\to \sA_g$$

\begin{cor}\label{cor:moduliofav}
    Given $q:\cA\to \sA_g$ as above,
 the following statement holds.
 
If
\[
  f:X\longrightarrow\mathcal A
\]
is a finite surjective morphism of degree $d$, where $X$ is smooth
and connected, then
\[
  (-1)^{g(g+1)/2}\chi(X)
  \geq
  d\,(-1)^{g(g+1)/2}\chi(\mathcal A).
\]
Consequently, if
$
  \chi(X)=\chi(\mathcal A),
$ then $f$ is an isomorphism. 
\end{cor}

\begin{rem}[Arbitrary base field of characteristic zero]\label{rem:base-field}
The main theorem above and its corollaries hold over any
algebraically closed field $k$ of characteristic zero, with $\chi$ defined by
$\ell$-adic \'etale cohomology. This follows easily from base change and Artin comparison.
\end{rem}

\begin{rem}[Positive characteristic]\label{rem:charp}
If $k$ is an algebraically closed field of characteristic $p$, then some hypothesis beyond finiteness of the morphism is likely needed. If $Y$ is defined over a finite
field of characteristic $p$, the $p$-power Frobenius is a finite surjective
self-morphism of degree $p^n$ which is a universal homeomorphism, so it leaves
$\chi(Y)$ unchanged and \eqref{eq:main-inequality} fails whenever
$(-1)^n\chi(Y)>0$. For separable morphisms of curves the inequality (at least for curves) does
survive: by Grothendieck--Ogg--Shafarevich the logarithmic Riemann--Hurwitz
formula acquires Swan conductor terms along the boundary, and these are
nonnegative, although the equality form requires tameness. In higher dimensions
our method does not apply, since no logarithmic Dubson-Kashiwara theorem is available for
$\ell$-adic sheaves, and one should expect $\operatorname{Ram}(f)$ to acquire
wild ramification contributions along the boundary.
\end{rem}

\noindent{\bf Sketch of the proofs} The proof of the theorem has two ingredients.  The first is the logarithmic
Dubson--Kashiwara index theorem of Wu--Zhou
\cite[Theorem~1.6]{WuZhou}, together with the positivity of
intersections of conical cycles in a nef vector bundle.  The first and the third authors used these results to show that every perverse sheaf $P$ on $Y$ has
nonnegative Euler characteristic under the hypotheses above
\cite[Lemma~2.2]{ArapuraPatel}.  Keeping track of the zero-section
component gives the sharper estimate
\[
  \chi(Y,P)\geq r(-1)^n\chi(Y),
\]
where $r$ is the generic rank of $P$.  The second ingredient is the
trace splitting
\[
  Rf_*\bbQ_X[n]\simeq \bbQ_Y[n]\oplus P_f.
\]
The trace-zero summand $P_f$ is perverse and has generic rank $d-1$.
The equality
\[
  \chi(Y,P_f)=(-1)^n\bigl(\chi(X)-\chi(Y)\bigr)
\]
then yields Theorem~\ref{thm:main}. The applications are proved by showing that the relevant spaces have appropriate compactifications with nef log cotangent bundles.\\

\noindent{\bf Organization:} In Section~\ref{sec:cc} we recall the log Dubson-Kashiwara formula and apply it to establish the sharper estimate
\[
  \chi(Y,P)\geq r(-1)^n\chi(Y),
\]
where $r$ is the generic rank of $P$
(Proposition~\ref{prop:rank-positivity}). In Section~\ref{sec:tracezero} we establish
the trace decomposition and the Euler-characteristic identity. These results are then used in Section~\ref{sec:proof-main} to deduce the main Theorem~\ref{thm:main} and
Corollaries~\ref{cor:degree-bound} and \ref{cor:rigidity}.
In Section~\ref{sec:semiabelian}, we prove the existence of a canonical compactification with nef log cotangent bundle for any semiabelian variety. We use this to deduce the result of Franecki-Kapranov (\cite{FraneckiKapranov}) for perverse sheaves on semiabelian varieties (\ref{cor:FK}) and
Corollary~\ref{cor:Alb}. In Section~\ref{sec:modulinefness} we prove nefness for log cotangent bundles of the moduli stack of stable curves. This is used in Section~\ref{sec:appstomoduli} to deduce
Corollaries~\ref{cor:moduliofcurves} and \ref{cor:moduliofav}.

\section{A lower bound for Euler characteristics of perverse sheaves}
\label{sec:cc}
Let $\oZ$ be a smooth projective variety of dimension $m$, $D \subset \oZ$ an SNC divisor, and set $Z := \oZ \setminus D$. Let $E:=\Omega^1_{\overline Z}(\log D)$ and $T^*_{\overline Z}(\log D)$ denote the associated logarithmic cotangent bundle with zero section $  s:\overline Z\longrightarrow E$. If $C\subset E$ is an $m$-dimensional conical
cycle, define
\[
  \langle C,\overline Z\rangle
  :=[C] \cdot [s(\oZ)].
\]
This is the intersection number of $C$ with the zero section.

We recall some basic properties of the characteristic cycle of constructible sheaves and the (logarithmic) Dubson-Kashiwara formula. Let $\rD_c^b(Z,\bbQ)$ denote the bounded derived category of constructible sheaves with rational coefficients. For $K\in \rD_c^b(Z,\bbQ)$, let $\CC(K)$ denote its characteristic cycle
in the group of codimension $m$-cycles $\cZ^m(T^*Z)$.  We use the following standard facts; see
\cite[Chapter~IX]{KashiwaraSchapira} and
\cite[Sections~4.1--4.2]{Dimca}.

\begin{enumerate}
\item Every irreducible component of $\CC(K)$ is a conical
Lagrangian subvariety of $T^*Z$ of dimension $m$.

\item If $P$ is perverse, then $\CC(P)$ is an effective cycle.

\item Let $i : Y \hookrightarrow Z$ be a smooth closed subvariety of dimension $n$. Suppose that  $P$ is the minimal perverse extension of $L[n]$, where
$L$ is a local system of rank $r$ on
a dense open subset $U\subset Y$. Then the coefficient 
of the conormal bundle $T^*_Y Z\subset T^*Z$ in $\CC(i_*P)$ is $r$.
\end{enumerate}

Thus a perverse sheaf $P$  on $Y$ of generic rank $r$ has a characteristic
cycle of the form
\begin{equation}
\label{eq:CC-decomposition}
  \CC(i_*P)
  =r[T^*_Y Z]+\sum_\alpha m_\alpha[\Lambda_\alpha],
  \qquad m_\alpha\in\mathbb Z_{>0},
\end{equation}
where the projection of every $\Lambda_\alpha$ is contained in a
proper closed subvariety of $Y$.

Let $\overline{\CC(P)}$ denote the closure of the characteristic cycle
inside $E$.  The logarithmic
Dubson--Kashiwara formula gives
\begin{equation}
\label{eq:log-DK}
  \chi(Z,P)
  =\bigl\langle \overline{\CC(P)},\overline Z\bigr\rangle.
\end{equation}
The form \eqref{eq:log-DK} is established by Wu--Zhou
\cite[Theorem~1.6]{WuZhou}; see also the formulation in
\cite[Section~2.2]{ArapuraPatel}.

We suppose now that $E$ is nef. In this setting, the Fulton--Lazarsfeld positivity theorem
\cite[Theorem~8.2.6]{Lazarsfeld} implies that
\begin{equation}
\label{eq:FL}
  \langle C,\overline Z\rangle\geq0
\end{equation}
for every $m$-dimensional conical subvariety $C\subset E$.  

\begin{lem}
\label{lem:zero-section}
Let $Y,Z,\oZ,D, E$ be as above, and assume that $E$ is nef.  Then
\begin{equation}
\label{eq:zero-section}
  \bigl\langle \overline{T^*_YZ} ,\overline Z\bigr\rangle
  =(-1)^n\chi(Y)\ge 0.
\end{equation}

\end{lem}

\begin{proof}
For the  equality, apply the logarithmic index formula
\eqref{eq:log-DK} and the remarks above to the (middle extension of the) perverse sheaf $\bbQ_Y[n]$. Note that since $Y \subset Z$ is closed, the minimal extension to $Z$ is just the usual push-forward. Its characteristic
cycle is $[T^*_YZ]$, while
\[
  \chi(Y,\bbQ_Y[n])=(-1)^n\chi(Y).
\]
The inequality follows from \eqref{eq:FL}.
\end{proof}

The following proposition is the rank-sensitive refinement of
the  nonnegativity lemma of \cite[Lemma~2.2]{ArapuraPatel}.

\begin{prop}
\label{prop:rank-positivity}
Suppose that $E$ is nef.
Let $P$ be a perverse sheaf on $Y$, and let $r$ be its generic rank;
that is, $P|_U\simeq L[n]$ on a dense open subset $U\subset Y$, for a
local system $L$ of rank $r$.  Then
\begin{equation}
\label{eq:rank-positivity}
  \chi(Y,P)\geq r(-1)^n\chi(Y).
\end{equation}
Consequently, if $P$ has full support (i.e. $r>0$) and 
$\chi(Y)\not=0$, then 
\begin{equation}
\label{eq:strict-positivity}
\chi(Y,P)>0.
\end{equation}
\end{prop}

\begin{proof}
Let $i :Y \hookrightarrow Z$ denote the closed embedding. Since $i$ is a closed immersion, $\chi(Y,P) = \chi(Z,i_*P)$. Moreover, it follows from the remarks above that 
\[
  \CC(i_*P)
  =r[T^*_YZ]
   +\sum_\alpha m_\alpha[\Lambda_\alpha].
\]
Taking closures in the logarithmic cotangent bundle in
\eqref{eq:CC-decomposition} gives
\[
  \overline{\CC(i_*P)}
  =r[\overline{T^*_YZ}]
   +\sum_\alpha m_\alpha[\overline\Lambda_\alpha].
\]
By the logarithmic index theorem, Lemma~\ref{lem:zero-section}, and
linearity of the Gysin intersection,
\begin{align*}
  \chi(Y,P)
  &=r(-1)^n\chi(Y)
    +\sum_\alpha m_\alpha
      \bigl\langle\overline\Lambda_\alpha,\overline Z\bigr\rangle.
\end{align*}
Each cycle $\overline\Lambda_\alpha$ is conical of dimension $m=\dim Z$.
Since $\Omega^1_{\overline Z}(\log D)$ is nef, every intersection in
the sum is nonnegative by \eqref{eq:FL}.  This proves
\eqref{eq:rank-positivity}.

If $P$ has full support on $Y$, then $r>0$. If $\chi(Y)\not=0$, then
$(-1)^n\chi(Y)>0$ by Lemma \ref{lem:zero-section}.
This inequality makes the first term on the right of the above equation strictly positive.
\end{proof}

\begin{rem}
For $r=0$, Proposition~\ref{prop:rank-positivity} reduces to the
nonnegativity statement of \cite[Lemma~2.2]{ArapuraPatel}.  The only
additional observation is that the conormal component $T_Y^*Z$
occurs with
multiplicity equal to the generic rank and that its contribution is
$(-1)^n\chi(Y)$.
\end{rem}

\section{The trace decomposition and the Euler-characteristic identity}\label{sec:tracezero}

We use the standard cohomological shift notation: for a complex $K$,
its shift $K[n]$ is characterized by
\[
  \mathcal H^j(K[n])=\mathcal H^{j+n}(K),
\]
where $\mathcal H^j(K)$ denotes the $j$-th cohomology sheaf of $K$.

Let $X$, $Y$, and $f:X\to Y$ be as in
Theorem~\ref{thm:main}. Since $f$ is finite and surjective,
\[
  n:=\dim X=\dim Y.
\]
Since $X$ and $Y$ are smooth of complex dimension $n$, the shifted
constant sheaves $\bbQ_X[n]$ and $\bbQ_Y[n]$ are perverse; see
\cite[\S 4]{BBD} and \cite[\S 2]{deCataldoMigliorini}. We denote by
$\Perv(Y,\bbQ)$ the abelian category of perverse sheaves of
$\bbQ$-vector spaces on $Y$.

Since $f$ is finite, the derived direct-image functor $Rf_*$ is
$t$-exact for the middle perverse $t$-structure. Consequently,
\[
Rf_*\bbQ_X[n] \in \Perv(Y,\bbQ).
\]
See
\cite[\S 4.1]{BBD}.

The following lemma is a particularly simple case of the decomposition theorem of Beilinson, Bernstein, Deligne, and Gabber \cite{BBD}. In the finite-map setting considered here, however, the required splitting can be obtained directly from the unit and trace maps.

\begin{lem}
\label{lem:trace-splitting}
There is a splitting in $\Perv(Y,\bbQ)$
\begin{equation}
\label{eq:trace-splitting}
  Rf_*\bbQ_X[n]\simeq \bbQ_Y[n]\oplus P_f,
\end{equation}
where $P_f$ is perverse. If $U\subset Y$ is a dense open subset over
which $f$ is \'{e}tale, then
\[
  P_f|_U\simeq L_0[n],
\]
where
\[
  L_0:=
  \ker\left(
    (f_*\bbQ_X)|_U\longrightarrow\bbQ_U
  \right)
\]
is the trace-zero local system of rank $d-1$. In particular, if
$d>1$, then $P_f$ has full support and generic rank $d-1$.
\end{lem}

\begin{proof}
There is a natural map
\[
  \eta:\bbQ_Y[n]\longrightarrow Rf_*\bbQ_X[n],
\]
induced by pulling locally constant functions on $Y$ back to $X$.
This is the unit map associated with the adjunction between $f^*$ and
$Rf_*$. Over the locus where $f$ is \'{e}tale, it is fiberwise the
diagonal map
\[
  a\longmapsto(a,\ldots,a).
\]

Since $f$ is finite, it is proper and Verdier duality commutes with
proper push-forward. Moreover, since $X$ and $Y$ are smooth of complex dimension
$n$, the perverse sheaves $\bbQ_X[n]$ and $\bbQ_Y[n]$ are Verdier
self-dual. Dualizing $\eta$ therefore gives a trace morphism
\[
  \tr:Rf_*\bbQ_X[n]\longrightarrow\bbQ_Y[n].
\]
See \cite{BBD} for Verdier duality and its compatibility with proper
direct image.

The unit and trace satisfy the standard identity
\[
  \tr\circ\eta=d\,\id_{\bbQ_Y[n]}.
\]
Over the dense \'{e}tale locus $U$, this is the elementary
calculation in which the diagonal map
\[
  a\longmapsto(a,\ldots,a)
\]
is followed by summation over the $d$ sheets.

Since we work with $\bbQ$-coefficients, the integer $d$ is invertible.
It follows that
\[
  (d^{-1}\tr)\circ\eta=\id_{\bbQ_Y[n]}.
\]
Thus $d^{-1}\tr$ is a retraction of $\eta$. Setting
\[
  P_f:=\ker(\tr)
\]
in the abelian category $\Perv(Y,\bbQ)$, we obtain the splitting
\eqref{eq:trace-splitting}.

Over $U$, the sheaf $(f_*\bbQ_X)|_U$ is a local system of rank $d$,
and $L_0$ is the kernel of summation onto the constant rank-one local
system. Hence $L_0$ has rank $d-1$, and
\[
  P_f|_U\simeq L_0[n].
\]
If $d>1$, then $P_f|_U$ is nonzero. Consequently, the support of
$P_f$ is a closed subset containing the dense open subset $U$, and
therefore
\[
  \Supp(P_f)=Y.
\]
Thus $P_f$ has full support and generic rank $d-1$.
\end{proof}

\begin{lem}
\label{lem:euler-Pf}
With the notation above,
\begin{equation}
\label{eq:euler-Pf}
  \chi(Y,P_f)
  =
  (-1)^n\bigl(\chi(X)-\chi(Y)\bigr).
\end{equation}
\end{lem}

\begin{proof}
Since taking derived global sections after derived direct image is
the same as taking derived global sections on the source,
\[
  R\Gamma\bigl(Y,Rf_*\bbQ_X[n]\bigr)
  \simeq
  R\Gamma(X,\bbQ_X[n]).
\]
Therefore,
\[
  \chi(Y,Rf_*\bbQ_X[n])
  =
  \chi(X,\bbQ_X[n])
  =
  (-1)^n\chi(X).
\]
On the other hand,
\[
  \chi(Y,\bbQ_Y[n])=(-1)^n\chi(Y).
\]
Taking Euler characteristics in the direct-sum decomposition
\eqref{eq:trace-splitting} gives
\[
  (-1)^n\chi(X)
  =
  (-1)^n\chi(Y)+\chi(Y,P_f),
\]
which proves \eqref{eq:euler-Pf}.
\end{proof}
\section{Proof of Theorem \ref{thm:main}: the Logarithmic Euler--Hurwitz inequality}
\label{sec:proof-main}

\begin{proof}[Proof of Theorem~\ref{thm:main}]
We apply Proposition~\ref{prop:rank-positivity} to the perverse sheaf
$P_f$ from Lemma~\ref{lem:trace-splitting}.  Its generic rank is
$d-1$, so
\[
  \chi(Y,P_f)
  \geq(d-1)(-1)^n\chi(Y).
\]
Using Lemma~\ref{lem:euler-Pf}, this becomes
\[
  (-1)^n\bigl(\chi(X)-\chi(Y)\bigr)
  \geq(d-1)(-1)^n\chi(Y).
\]
Rearranging gives
\[
  (-1)^n\chi(X)\geq d(-1)^n\chi(Y).
\]

For the more precise formula in terms of characteristic cycles, let $i:Y\hookrightarrow Z$ denote the inclusion and write
\[
  \CC(i_*P_f)
  =(d-1)[T^*_Y Z]
   +\sum_\alpha m_\alpha[\Lambda_\alpha].
\]
The proof of Proposition~\ref{prop:rank-positivity} gives
\[
  \chi(Y,P_f)
  =(d-1)(-1)^n\chi(Y)
   +\sum_\alpha m_\alpha
     \bigl\langle\overline\Lambda_\alpha,\overline Z\bigr\rangle.
\]
Substituting \eqref{eq:euler-Pf} and rearranging yields
\eqref{eq:main-formula}.  Every summand is nonnegative by \eqref{eq:FL}.
\end{proof}

\begin{proof}[Proof of Corollary~\ref{cor:degree-bound}]
Divide \eqref{eq:main-inequality} by the positive number
$(-1)^n\chi(Y)$.
\end{proof}

\begin{proof}[Proof of Corollary~\ref{cor:rigidity}]
Set $a=(-1)^n\chi(Y)>0$.  Since $\chi(X)=\chi(Y)$,
Theorem~\ref{thm:main} gives
\[
  a\geq da.
\]
Thus $d\leq1$.  Since $f$ is surjective, $d\geq1$, so $d=1$.  A
finite degree-one morphism is birational.  The target $Y$
is smooth, hence normal, and therefore a finite birational morphism to
$Y$ is an isomorphism.  If $f$ is a homotopy equivalence, then
$\chi(X)=\chi(Y)$.
\end{proof}

\begin{rem}[Equality]
If $f$ is finite \'{e}tale, then $P_f$ is a shifted local system on all of
$Y$, so its characteristic cycle has no components other than the
zero section.  Consequently, $\Ram(f)=0$, in agreement with the usual
multiplicativity $\chi(X)=d\chi(Y)$.  Under mere nefness, the converse
need not follow from the argument: a nonzero conical component can in
principle have zero intersection with the zero section.  An equality
criterion would require additional positivity or geometric input.
\end{rem}

\section{Proof of Corollary \ref{cor:Alb}: Semiabelian varieties}\label{sec:semiabelian}
In this section, we give a proof of Corollary \ref{cor:Alb}. Along the way we obtain a simple proof of a result of Franecki-Kapranov on the positivity of Euler characteristics of perverse sheaves on semiabelian varieties as an application of the log Dubson-Kashiwara formula. The key result is the following proposition which gives a canonical compactification of a semiabelian variety with nef log cotangent bundle.

\begin{prop}\label{prop:semiab}
    A semiabelian variety $Z$ admits a smooth projective compactification $\oZ$ such 
    that $D= \oZ\setminus Z$ is an SNC divisor and $\Omega_{\oZ}^1(\log D)$ is nef. 
\end{prop}

\begin{proof}
   We use the description of semiabelian varieties found in \cite[\S 5]{FraneckiKapranov}.
Given an abelian variety $A$ and line bundles $L_1,\ldots, L_N\in Pic^0(A)$, then
$$Z = (L_1^*)\times_A\ldots \times_A (L_N^*)$$
 with an appropriate group law is a semiabelian variety, where $L_i^*$ is the $\bbC^*$-bundle obtained by removing the zero section.
Conversely any semiabelian variety is of this form. Using this description, we can see that
$$\oZ = \bbP(\cO \oplus L_1^{-1})\times_A\ldots \times_A \bbP(\cO\oplus L_N^{-1})$$
is a smooth compactification such that $D= \oZ\setminus Z$ is an SNC divisor.

It remains to show that $\Omega_{\oZ}^1(\log D)$ is nef. 
Let $f:\oZ\to A$, $f_i:\bbP(\cO\oplus L_i^{-1})\to A$ and $\pi_i:\oZ\to \bbP(\cO\oplus L_i^{-1})$ denote the projections.
Then we have an exact sequence
$$0\to f^*\Omega_A^1\to \Omega_{\oZ}^1(\log D) \to \bigoplus_i \pi_i^*\Omega^1_{\bbP(\cO\oplus L_i^{-1})/A}(D^i_0+D^i_\infty)\to 0$$
where $D_0^i, D_\infty^i$ are the zero and infinity sections of $\bbP(\cO\oplus L_i^{-1})$ (i.e. corresponding to the summands $\cO$ and $L_i^{-1}$).
Since $\Omega_A^1$ is trivial, and the class of nef bundles is closed under extensions, it suffices to prove that each factor on the right is nef. Note that, in our setting, the relative cotangent bundle is rank 1 and therefore the usual relative canonical bundle.
We have that 
$$\cO(D_0^i) = \cO_{\bbP(\cO\oplus L_i^{-1})}(1)$$
$$\cO(D_\infty^i)= \cO_{\bbP(\cO\oplus L_i^{-1})}(1)\otimes f_i^*L_i $$
Putting these together with the standard formula for the relative canonical bundle of a projective bundle,
 we obtain
 $$ \Omega^1_{\bbP(\cO\oplus L_i^{-1})/A}(D^i_0+D^i_\infty) = \left[f_i^*{\wedge}^2(\cO\oplus L_i^{-1})\otimes \cO(-2)\right]\otimes \cO(2)\otimes f_i^*L_i = \cO_{\bbP(\cO\oplus L_i^{-1})}$$
 This is certainly nef.
 
\end{proof}

As a corollary, we obtain a simple proof  of one of  the main results of  \cite{FraneckiKapranov}.

\begin{cor}[Franecki-Kapranov]\label{cor:FK}
    If $P$ is a perverse sheaf on a semiabelian variety $Z$,
    $$\chi(Z,P)\ge 0.$$
    We have strict inequality if $P$ has full support on a smooth closed subvariety with nonzero Euler characteristic.
\end{cor}

\begin{proof}
    The first part follows immediately from Proposition \ref{prop:semiab}, the log Dubson-Kashiwara formula, and Fulton-Lazarsfeld positivity. The second part follows from \ref{prop:rank-positivity}. 
\end{proof}

\begin{proof}[Proof of Corollary~\ref{cor:Alb}]
The corollary is an immediate consequence of  Theorem \ref{thm:main} and Proposition \ref{prop:semiab}.
\end{proof}

\section{Nefness of log cotangent bundle of moduli of curves}\label{sec:modulinefness}

We remind the reader that all varieties and stacks are over $\mathbb{C}$. When $2g-2+n>0$, we let
 $\sM_{g,n}$ (resp. $\overline{\sM_{g,n}}$) denote the moduli stack of genus-$g$ smooth curves with $n$ marked points (resp. its Deligne-Mumford-Knudsen compactification \cite{Knudsen}). 
This is a smooth DM stack of dimension $3g-3+n$.
Let 
$$\Delta_{g,n}= \csM_{g,n}-\sM_{g,n}$$
denote the divisor at infinity. We simply write $\Delta$ if there is no confusion.
The following fact is stated in \cite[remark 6.15]{Litt} without proof, and the goal in this section is to supply one.

\begin{thm}\label{thm:Omeganef}
If $2g-2+n>0$, then   $\Omega^1_{\csM_{g,n}}(\log \Delta)$ is nef.
\end{thm}

\begin{defn}[nef bundle on a DM stack; {\cite[Def.~1.6]{Litt}}]\label{def:nef}
Let $\mathcal Y$ be a Deligne--Mumford stack of finite type over $\mathbb C$
and $\cE$ a vector bundle on $\mathcal Y$. We say $\cE$ is \emph{nef} if for
every morphism $h\colon B\to\mathcal Y$ from a smooth projective
\emph{curve} $B$ and every quotient line bundle $h^*\cE\twoheadrightarrow\cL$
on $B$ one has $\deg_B\cL\ge 0$.
\end{defn}

\begin{rem}
The test object $B$ in Definition~\ref{def:nef} is a \emph{scheme}, and
$h^*\cE$ is an ordinary vector bundle on it, so ``nef'' is a condition
about schemes. Nefness is stable under pullback and, for a finite
surjective morphism $B'\to B$ of smooth projective curves, $h^*\cE$ is nef
if and only if  its pullback to $B'$ is nef \cite[Prop.~6.1.7(iv)]{Lazarsfeld}. For a
vector bundle on a smooth projective scheme, Definition~\ref{def:nef}
agrees with nefness of $\cO_{\mathbb P(\cE)}(1)$ \cite[Thm.~6.2.12]{Lazarsfeld}.
\end{rem}

Let $\pi:\overline{\cC_g} \to \csM_g$ denote the universal curve, and let $\omega_\pi$ denote the relative
dualizing sheaf.

\begin{lem}\label{lem:identify}
Let $g \geq 2$.
There is a canonical isomorphism of vector bundles on $\csM_g$
\begin{equation}\label{eq:identify}
\Omega^1_{\csM_g}(\log\Delta)\;\cong\;\pi_*\bigl(\omega_\pi^{\otimes 2}\bigr).
\end{equation}
Moreover $\pi_*\omega_\pi^{\otimes 2}$ is locally free of rank $3g-3$ and
its formation commutes with arbitrary base change.
\end{lem}

\begin{proof}
The isomorphism follows from the
application of logarithmic Kodaira-Spencer Theory \cite[Prop. 3.14]{Kato} and the interpretation of $\csM_g$ in terms of log geometry \cite{FKato}.
Since for any stable curve $C$ in $\csM_g$, we have  $\dim \rH^0(C, \omega_C^{\otimes 2})=3g-3$ and
$\rH^1(C,\omega_C^{\otimes 2})=0$, the remaining statement follows.

\end{proof}

\begin{prop}\label{prop:nef}
$\Omega^1_{\csM_g}(\log\Delta)$ is nef.
\end{prop}
\begin{proof}
Let $h\colon B\to\overline{\sM}_g$ be a morphism from a smooth projective
curve. By Definition~\ref{def:nef} it suffices to show that
$h^*\Omega^1_{\overline{\sM}_g}(\log\Delta)$ is nef on $B$.

Set $X:=B\times_{\overline{\sM}_g}\overline{\mathcal C}_g$ with projection
$f\colon X\to B$. The morphism $\pi$ is representable, so $X$
is a scheme, projective over $B$; thus $f$ is a flat projective family of stable
curves over $B$. Since the dualizing sheaf commutes with the base change $h$,
Lemma~\ref{lem:identify} gives
\begin{equation}\label{eq:pullback}
  h^*\Omega^1_{\overline{\sM}_g}(\log\Delta)
   \;\simeq\;h^*\pi_*\omega_\pi^{\otimes2}
   \;\simeq\;f_*\bigl(\omega_{X/B}^{\otimes2}\bigr).
\end{equation}

We check the hypotheses of \cite[Theorem 1.7]{FujinoModuli} for $f\colon X\to B$.
The fibres of $f$ are nodal, hence Gorenstein, so $X$ is Gorenstein and in
particular satisfies Serre's condition $S_2$; moreover $\omega_{X/B}$ is a line
bundle. Along the locus where a node persists in the fibres, $X$ is analytically
$\{xy=0\}\times B$, and at an isolated node it is the surface singularity
$\{xy=s^k\}$, which is of codimension two in $X$; hence $X$ is normal crossing in
codimension one, and its singularities are normal crossing or Du Val, so in
either case semi-log-canonical. Since $f$ is flat with one-dimensional fibres and
$B$ is a smooth curve, every irreducible component of $X$ has dimension two and
is therefore dominant onto $B$. Finally $\omega_{X/B}^{\otimes3}$ is $f$-very
ample, so it is $f$-generated.

By \cite[Theorem 1.7]{FujinoModuli}, $f_*\omega^{\otimes m}_{X/B}$ is nef on $B$
for every $m\ge1$. Taking $m=2$ and using \eqref{eq:pullback} completes the proof.
\end{proof}

To finish the proof of theorem \ref{thm:Omeganef}, we need the following lemmas.

\begin{lem}\label{lem:induction-n}
Suppose that $2g-2+n>0$. If $\Omega^1_{\overline{\sM}_{g,n}}(\log\Delta)$
is nef, then $\Omega^1_{\overline{\sM}_{g,n+1}}(\log\Delta)$ is nef.
\end{lem}
 
\begin{proof}
Let $\pi\colon\overline{\sM}_{g,n+1}\to\overline{\sM}_{g,n}$ be the
morphism given by forgetting the last point and stabilizing (\cite{Knudsen}). Note that this identifies
$\overline{\sM}_{g,n+1}$ with the universal curve over
$\overline{\sM}_{g,n}$. Let $\sigma_1,\dots,\sigma_n$ be its
tautological sections. Write $\Delta_{g,n}$ and $\Delta_{g,n+1}$ for the
respective boundaries. Moreover, we have the following relations for the boundary divisors (\cite[Proof of 2.7]{Knudsen}):
\[
  \Delta_{g,n+1}\;=\;\pi^*\Delta_{g,n}+\sum_{i=1}^n\sigma_i .
\]
Since $\pi$ is log smooth for the log structures determined by these boundaries,
there is an exact sequence of locally free sheaves
\begin{equation}\label{eq:Omegaseq}
  0\to\pi^*\Omega^1_{\overline{\sM}_{g,n}}(\log\Delta_{g,n})
   \to\Omega^1_{\overline{\sM}_{g,n+1}}(\log\Delta_{g,n+1})
   \to\Omega^1_\pi(\log)\to0 ,
\end{equation}
whose right-hand term has rank one (since the morphism is log smooth of relative dimension 1). Taking determinants in
\eqref{eq:Omegaseq} and using the identity above we obtain:
\[
  \Omega^1_\pi(\log)
  \;\simeq\;\omega_{\overline{\sM}_{g,n+1}}(\Delta_{g,n+1})
     \otimes\pi^*\bigl(\omega_{\overline{\sM}_{g,n}}(\Delta_{g,n})\bigr)^{-1}
  \;\simeq\;\omega_\pi\Bigl(\sum_{i=1}^n\sigma_i\Bigr),
\]
 The sheaf on the left of
\eqref{eq:Omegaseq} is nef by hypothesis, and the sheaf on the right is known to be nef
(\cite[Theorem 0.4]{Keel}). Since an extension of nef bundles is nef \cite{Lazarsfeld},
the middle term is nef.
\end{proof}

\begin{lem}\label{lem:base-n}
  $\Omega^1_{\csM_{g,n}}(\log \Delta)$ is nef, when $(g,n)=(0,3)$ or $(g,n)=(1,1)$  
\end{lem}

\begin{proof}
    This is trivially true for the first case, and standard for the second.
\end{proof}

\begin{proof}[Proof of theorem \ref{thm:Omeganef}]
 This follows from proposition \ref{prop:nef}, lemmas \ref{lem:base-n}, \ref{lem:induction-n} and induction on $n$.
\end{proof}

As an application of this, we give a slight refinement of the following theorem \cite[cor. 7.6]{ArapuraPatel}, with a new and simpler proof.

\begin{thm}\label{thm:strictpositivityMgn}
    If $2g-2+n>0$, then for any perverse sheaf $P$ on $\sM_{g,n}^{an}$, the orbifold Euler characteristic $\chi^{\mathrm{orb}}(\sM_{g,n}^{an},P)\ge 0$. Furthermore, the inequality is strict  if the support of $P$ is Zariski dense.
\end{thm}

\begin{proof}
We pass to the level $\lambda$ so that the moduli space $\sM_{g,n,\lambda}$ is a scheme and consider its Boggi--Pikaart (\cite{BoggiPikaart}) compactification $\csM_{g,n,\lambda}$. The `forget the level structure' morphism 
$$\sM_{g,n,\lambda} \rightarrow \sM_{g,n}$$ is etale and extends to a log-etale morphism on the compactifications (see section \ref{sec:appstomoduli} for details). It follows that the log cotangent bundle on $\csM_{g,n,\lambda}$ is the pullback of the log cotangent bundle on $\csM_{g,n}$ and therefore, by Theorem \ref{thm:Omeganef}, it is nef. Since the level $\lambda$ moduli spaces are schemes we may apply the log Dubson--Kashiwara formula and Fulton--Lazarsfeld positivity (as in \cite[lemma 2.2]{ArapuraPatel}) to conclude that every perverse sheaf on $\sM_{g,n,\lambda}$ has nonnegative Euler characteristic. Since this is a finite etale cover of $\sM_{g,n}$, one obtains the analogous assertion for the orbifold Euler characteristic of perverse sheaves on $\sM_{g,n}$. 

The second part of the statement also holds for $\cM_{g,n,\lambda}$ as a consequence of Proposition \ref{prop:rank-positivity} and the fact that ( \cite{HarerZagier} or \cite{McMullen}) $(-1)^N\chi^{\mathrm{orb}}(\cM_{g,n})>0$ with $N = \dim(\cM_{g,n})$ (and therefore $(-1)^N\chi(\cM_{g,n,\lambda})>0$).
One deduces the statement for $\cM_{g,n}$ by once again descending via this etale cover. 
\end{proof}

\section{Proofs of Corollaries \ref{cor:moduliofcurves} and \ref{cor:moduliofav} }\label{sec:appstomoduli}

\begin{proof}[Proof of Corollary~\ref{cor:moduliofcurves}]
First form the canonical compactified level space associated with
$q$ by normalizing $\csM_{g,n}$ in the function field of
$\mathcal M$. Thus $q$ extends to a finite compactified level
morphism.

Applying \cite[Corollary~2.10]{BoggiPikaart} to the corresponding
finite compactified cover, we obtain a connected smooth finite
Galois compactified level cover $\overline{\mathcal M'}$ dominating
it. Restricting the dominating map to the inverse image of the open
moduli stack gives a commutative diagram
$$
\xymatrix{
 \cM'\ar[r]^{q'}\ar[rd]^{p} & \cM\ar[d]^{q} \\ 
  & \sM_{g,n}
}
$$

By the definition of domination for level structures, $q'$ is a
finite \'{e}tale cover. In particular, $\mathcal M'$ is smooth and
connected.

The map $p=q \circ q'$ induces a
finite morphism
\[
  \overline p:\overline{\mathcal M'}
  \longrightarrow\csM_{g,n}.
\]
Since $\overline{\mathcal M'}$ is finite over the projective coarse
moduli space $\overline M_{g,n}$, it is projective.

Set
\[
  E=\overline{\mathcal M'}\setminus\mathcal M',
  \qquad
  \Delta=\csM_{g,n}\setminus\mathscr M_{g,n}.
\]
The local monodromy description of the compactified level cover shows
that $E$ is a normal-crossing divisor. Moreover,
\cite[Proposition~2.1]{BoggiPikaart} rules out self-intersections of
its irreducible components. Hence $E$ is an SNC divisor.

The compactified level morphism is logarithmically \'{e}tale along
the boundary. More precisely, \'{e}tale locally near a boundary
point, one can choose coordinates such that
\[
  t_i=z_i^{e_i}
  \quad (1\leq i\leq r),
  \qquad
  t_j=z_j
  \quad (r<j\leq 3g-3+n),
\]
where $t_1\cdots t_r=0$ and $z_1\cdots z_r=0$ are local equations for
$\Delta$ and $E$, respectively. Consequently,
\[
  \overline p^{\,*}\!\left(\frac{dt_i}{t_i}\right)
  =
  e_i\frac{dz_i}{z_i}
  \quad (1\leq i\leq r),
  \qquad
  \overline p^{\,*}(dt_j)=dz_j
  \quad (r<j\leq 3g-3+n).
\]
Since each $e_i$ is nonzero over $\mathbb C$, pullback induces an
isomorphism
\[
  \Omega^1_{\overline{\mathcal M'}}(\log E)
  \simeq
  \overline p^{\,*}
  \Omega^1_{\csM_{g,n}}(\log\Delta).
\]
By Theorem~\ref{thm:Omeganef},
\[
  \Omega^1_{\csM_{g,n}}(\log\Delta)
\]
is nef. Since the pullback of a nef vector bundle is nef, it follows
that
\[
  \Omega^1_{\overline{\mathcal M'}}(\log E)
\]
is nef.

Let $e=\deg(p)$. Since $p$ is a finite \'{e}tale cover in the
orbifold sense, multiplicativity of orbifold Euler characteristic
gives
\[
  \chi(\mathcal M')
  =
  e\,\chi^{\mathrm{\mathrm{orb}}}(\mathscr M_{g,n}).
\]
By the Harer--Zagier formula \cite{HarerZagier},
\[
  (-1)^{3g-3+n}\chi^{\mathrm{orb}}(\mathscr M_{g,n})>0.
\]
Since $e>0$, we conclude that
\[
  (-1)^{3g-3+n}\chi(\mathcal M')>0.
\]
It follows that
\[
  (-1)^{3g-3+n}\chi(\mathcal M)>0.
\]

Now let
\[
  f:X\longrightarrow\mathcal M
\]
be a finite surjective morphism of degree $d$, with $X$ smooth. 
Let
$$ f':X' = X\times_{\cM} \cM'\to \cM'$$
Applying Theorem~\ref{thm:main} to the smooth
quasiprojective variety $\mathcal M'$ and its compactification
$\overline{\mathcal M'}$ gives
\[
  (-1)^{3g-3+n}\chi(X')
  \geq
  d\,(-1)^{3g-3+n}\chi(\mathcal M').
\]
By dividing both sides by $\deg q'$, we obtain
\[
  (-1)^{3g-3+n}\chi(X)
  \geq
  d\,(-1)^{3g-3+n}\chi(\mathcal M).
\]
Finally, if $\chi(X)=\chi(\cM)$, then the above inequality implies
$d=1$. Since $f$ is finite of degree one and $\cM$ is smooth, hence
normal, it follows that $f$ is an isomorphism.
\end{proof}

\begin{proof}[Proof of corollary \ref{cor:moduliofav}]
The proof is very similar to the previous argument, so we just outline the main steps.
We may choose a smooth toroidal compactification $\cA\subset {\bar \cA}$ such that the complement is an SNC divisor $D$ \cite{FaltingsChai}. By \cite[lemma 3.3] {ArapuraPatel}, $\Omega^1_{\bar \cA}(\log D)$ is nef. Furthermore, a theorem of Harder \cite{Harder} implies that $(-1)^{g(g+1)/2}\chi(\cA)>0$. The corollary now follows from Theorem \ref{thm:main} and Corollary \ref{cor:rigidity}.

\end{proof}

\end{document}